\documentclass{amsart}
\usepackage{amssymb}
\usepackage[mathcal]{eucal}
\usepackage{upref}

\usepackage{hyperref}

\usepackage[lite]{amsrefs}

\usepackage[all,cmtip]{xy}

\DeclareMathAlphabet{\EuRm}{U}{eur}{m}{n}
\SetMathAlphabet{\EuRm}{bold}{U}{eur}{b}{n}

\newtheoremstyle{slplain}
 {.5\baselineskip\@plus.2\baselineskip\@minus.2\baselineskip}
 {.5\baselineskip\@plus.2\baselineskip\@minus.2\baselineskip}
 {\slshape}
 {}
 {\bfseries}
 {.}
 { }
 {}

\numberwithin{equation}{section}

\theoremstyle{slplain}
\newtheorem{thm}[equation]{Theorem}  
     
\newtheorem{lem}[equation]{Lemma}         
\newtheorem{prop}[equation]{Proposition}  

\theoremstyle{definition}

\theoremstyle{remark}
\newtheorem{notation}[equation]{Notation}

\newcommand{\thmref}{Theorem~\ref}
\newcommand{\propref}{Proposition~\ref}

\newcommand{\notref}{Notation~\ref}

\newcommand{\secref}{Section~\ref}

\newcommand{\iso}{\approx}

\DeclareMathOperator*{\colim}{colim}

\DeclareMathOperator{\Map}{Map}
\newcommand{\pullback}[3]{#1\mathbin{\mathord{\times}_{\!#2}}#3}
\newcommand{\mappullback}[3]{\pullback{\Map#1}{\Map#2}{\Map#3}}

\newcommand{\Btilde}{\widetilde{B}}

\newcommand{\ftilde}{\tilde{f}}
\newcommand{\gtilde}{\tilde{g}}

\newcommand{\Bhat}{\widehat{B}}

\newcommand{\CW}{\mathrm{CW}}

\newcommand{\Period}{\rlap{\enspace .}}

\begin{document}

\title{Functorial CW-approximation}

\author{Philip S. Hirschhorn}

\address{Department of Mathematics\\
   Wellesley College\\
   Wellesley, Massachusetts 02481}

\email{psh@math.mit.edu}

\urladdr{http://www-math.mit.edu/\textasciitilde psh}

\date{August 8, 2015}

\begin{abstract}
  The usual construction of a CW-approximation is functorial up to
  homotopy, but it is not functorial.  In this note, we construct a
  functorial CW-approximation.  Our construction takes inclusions of
  subspaces into inclusions of subcomplexes, and commutes with
  intersections of subspaces of a fixed space.
\end{abstract}

\maketitle

\tableofcontents

\section{Introduction}
\label{sec:intro}

A \emph{CW-approximation} to a topological space $B$ is a CW-complex
$\Btilde$ together with a weak equivalence $\Btilde \to B$.  The usual
construction of a CW-approximation is functorial up to homotopy, but
it is not functorial.  In this note, we construct a functorial
CW-approximation.  Our construction takes inclusions of subspaces into
inclusions of subcomplexes (see \thmref{thm:subspace}), and commutes
with intersections of subspaces of a fixed space (see
\thmref{thm:Intersect}).

We construct a CW-approximation to a space using a construction that
functorially factors a map $A \to B$ as $A \to \Btilde \to B$ where $A
\to \Btilde$ is a relative CW-complex and $\Btilde \to B$ is a weak
equivalence; applying this to the map $\emptyset \to B$ produces a
CW-approximation $\Btilde \to B$ to $B$.

We actually define two such factorizations.  The first is for
arbitrary maps $A \to B$ (see \thmref{thm:cwapprox}).  If $A$ is a
nonempty CW-complex, though, then the relative CW-complex $A \to
\Btilde$ that it produces will not, in general, be the inclusion of a
subcomplex.  Thus, we construct a different functorial factorization
in \thmref{thm:subcomp} for maps $A \to B$ in which $A$ is a
CW-complex; in the factorization $A \to \Btilde \to B$ that it
produces, the relative CW-complex $A \to \Btilde$ is the inclusion of
a subcomplex.

We show in \thmref{thm:subspace} that if the factorization of
\thmref{thm:cwapprox} is used to construct a functorial
CW-approximation (by factoring the maps with domain the empty space),
then this construction turns an inclusion of a subspace into an
inclusion of a subcomplex, i.e., if $B$ is a subspace of $B'$, then
$\Btilde$ is a subcomplex of $\Btilde'$.  Thus, it defines a
functorial CW-approximation for pairs, triads, etc.  We also show that
this operation commutes with taking intersections of subspaces of a
fixed space (see \thmref{thm:Intersect}).

\section{The main theorems}
\label{sec:theorem}

\subsection{The first factorization}
\label{sec:FirstFact}

\begin{thm}
  \label{thm:cwapprox}
  Every map $f\colon A \to B$ has a functorial factorization $A
  \xrightarrow{j} \Btilde \xrightarrow{p} B$ such that $j$ is a relative
  CW-complex and $p$ is a weak equivalence.
\end{thm}

To obtain a CW-approximation $\Btilde \to B$ to a space $B$, you apply
the factorization of \thmref{thm:cwapprox} to the map $\emptyset \to
B$.  We show in \thmref{thm:subspace} that if $B$ is a subspace of
$B'$ then the map $\Btilde \to \Btilde'$ is the inclusion of a
subcomplex, and we show in \thmref{thm:Intersect} that this operation
commutes with taking intersections of subspaces of a fixed space.

The outline of the proof of \thmref{thm:cwapprox} follows that of the
standard construction of a CW-approximation, but instead of choosing
maps of spheres that represent elements of homotopy groups to be
killed by attaching disks, we attach disks using all possible such
maps.  Thus, we attach many more cells than are required, but the
result is that our construction is functorial.

This construction is a cross between the usual construction of a
functorial-only-up-to-homotopy CW-approximation to a space and the
small object argument used to factorize maps in model categories
(\cite{MCL}*{Prop.~10.5.16}).  The standard small object argument
would produce a factorization into a relative cell complex (in which
the attaching maps of cells do not, in general, factor through a
subspace of lower dimensional cells) followed by a map that is both a
weak equivalence and a fibration; our construction produces a relative
CW-complex followed by a weak equivalence.  The proof of
\thmref{thm:cwapprox} is in \secref{sec:proof}.

\subsection{The second factorization}
\label{sec:SecondFact}

If the space $A$ is nonempty, then even if it is a CW-complex, the
space $\Btilde$ produced by \thmref{thm:cwapprox} will not generally
be a CW-complex, because there is no restriction on how the cells
attached to construct $\Btilde$ out of $A$ meet the cells of $A$.
Thus, we will also prove the following theorem.

\begin{thm}
  \label{thm:subcomp}
  Every map $f\colon A \to B$ such that $A$ is a CW-complex has a
  functorial factorization $A \xrightarrow{j} \Btilde \xrightarrow{p}
  B$ such that $j$ is the inclusion of a subcomplex of a CW-complex
  and $p$ is a weak equivalence, where ``functorial'' means that it is
  natural with respect to diagrams
  \begin{displaymath}
    \xymatrix{
      {A} \ar[r]^{f} \ar[d]
      & {A'} \ar[d]\\
      {B} \ar[r]_{g}
      & {B'}
    }
  \end{displaymath}
  in which $f\colon A \to A'$ is a cellular map of CW-complexes.
\end{thm}

\thmref{thm:subcomp} can also be used to obtain a functorial
CW-approximation to a space $B$ by applying it to the map $\emptyset
\to B$, but we show in \propref{prop:SameCW} that this produces the
same result as using \thmref{thm:cwapprox}.

The proof of \thmref{thm:subcomp} is in \secref{sec:subcomp}.

\begin{prop}
  \label{prop:SameCW}
  If $\Btilde \to B$ is the CW-approximation to $B$ obtained by
  applying the factorization of \thmref{thm:cwapprox} to $\emptyset
  \to B$ and $\Bhat \to B$ is the CW-approximation to $B$ obtained by
  applying the factorization of \thmref{thm:subcomp} to $\emptyset \to
  B$, then there is a natural isomorphism $\Bhat \to \Btilde$ that
  makes the diagram
  \begin{displaymath}
    \xymatrix@R=1ex{
      {\Bhat} \ar[dr] \ar[dd]\\
      & {B}\\
      {\Btilde} \ar[ur]
    }
  \end{displaymath}
  commute.
\end{prop}

The proof of \propref{prop:SameCW} is in \secref{sec:SameCW}.

\subsection{Relative CW-approximation}
\label{sec:relaprox}

The constructions of \thmref{thm:cwapprox} and \thmref{thm:subcomp}
can be used to create relative CW-approximations.
\begin{thm}
  \label{thm:subspace}
  If $(B',B)$ is a pair of spaces (i.e., if $B$ is a subspace of the
  space $B'$) then in the commutative square
  \begin{displaymath}
    \xymatrix{
      {\Btilde} \ar[r]^{\ftilde} \ar[d]
      & {\Btilde'} \ar[d]\\
      {B} \ar[r]_{f}
      & {B'}
    }
  \end{displaymath}
  obtained by applying the factorization of \thmref{thm:cwapprox} to
  the maps $\emptyset \to B$ and $\emptyset \to B'$, the map
  $\ftilde\colon \Btilde \to \Btilde'$ is an inclusion of a
  subcomplex.
\end{thm}

Thus, \thmref{thm:cwapprox} creates relative CW-approximations for
pairs, triads, etc.  Alternatively, given a pair $(B',B)$, one could
apply \thmref{thm:subcomp} to the map $\emptyset \to B$ to obtain
$\Btilde \to B$ and then apply \thmref{thm:subcomp} to the composition
$\Btilde \to B \to B'$ to obtain $\Btilde' \to B'$, and $\Btilde$
would be a subcomplex of $\Btilde'$.  The proof of
\thmref{thm:subspace} is in \secref{sec:subspace}.

\begin{thm}[CW-approximation commutes with intersections]
  \label{thm:Intersect}
  If $X$ is a space, let $\CW(X)$ denote the CW-complex obtained by
  applying the factorization of \thmref{thm:cwapprox} to the map
  $\emptyset \to X$.  If $X$ is a space, $S$ is a set, and for every
  element $s$ of $S$ we have a subspace $X_{s}$ of $X$, then each
  $\CW(X_{s})$ is a subcomplex of $\CW(X)$, and
  \begin{displaymath}
    \bigcap_{s\in S} \CW(X_{s}) =
    \CW\Bigl(\bigcap_{s\in S} X_{s}\Bigr) \Period 
  \end{displaymath}
\end{thm}

The proof of \thmref{thm:Intersect} is in \secref{sec:Intersect}.

\section{The proof \thmref{thm:cwapprox}}
\label{sec:proof}

We construct the factorization in \secref{sec:construct}, show that
the map $\Btilde \to B$ is a weak equivalence in \secref{sec:weq}, and
show that the construction is functorial in \secref{sec:func}.

\subsection{The construction}
\label{sec:construct}

We will construct a sequence of spaces
\begin{displaymath}
  \xymatrix{
    *+[l]{A = A_{-1}} \ar[r] \ar[d]
    & {A_{0}} \ar[r] \ar[dl]
    & {A_{1}} \ar[r] \ar[dll]
    & {\cdots}\\
    {B}
  }
\end{displaymath}
that map to $B$ and then let $\Btilde = \colim_{n}A_{n}$.  Each
$A_{n}$ for $n \ge 0$ will be constructed from $A_{n-1}$ by attaching
$n$-cells in such a way that the map $A_{n} \to B$ is $n$-connected
(see \notref{not:connected}).  Since spheres and disks are compact,
any map from a sphere or disk to $\colim_{n}A_{n}$ will factor through
some $A_{n}$, and so we will have $\pi_{i}\Btilde =
\colim_{n}\pi_{i}A_{n}$ for all $i \ge 0$, and the map $\Btilde \to B$
will be a weak equivalence.

We begin by letting $A_{-1} = A$, and then defining
\begin{displaymath}
  A_{0} = A_{-1} \amalg \Bigl(\coprod_{D^{0} \to B} D^{0}\Bigr)
  \Period
\end{displaymath}
That is, we let $A_{0}$ be the coproduct of $A_{-1}$ with a single
point for each map of a point to $B$; this maps to $B$ by taking the
$D^{0}$ indexed by a map $D^{0} \to B$ to $B$ by that indexing map.

To construct $A_{1}$ we construct the pushout
\begin{displaymath}
  \xymatrix@=1em{
    {\hspace{-3em}
      \coprod_{\mappullback{(S^{0},A_{0})}{(S^{0},B)}{(D^{1},B)}}
      \hspace{-4.5em}S^{0}}
    \ar[rr] \ar[dd]
    && {A_{0}} \ar@{..>}[dl] \ar[dd]\\
    & {A_{1}} \ar@{..>}[dr]\\
    {\hspace{-3em}
      \coprod_{\mappullback{(S^{0},A_{0})}{(S^{0},B)}{(D^{1},B)}}
      \hspace{-4.5em}D^{1}}
    \ar@{..>}[ur] \ar[rr]
    && {B}
  }
\end{displaymath}
where $\mappullback{(S^{0},A_{0})}{(S^{0},B)}{(D^{1},B)}$ is the set
of commutative squares
\begin{displaymath}
  \xymatrix{
    {S^{0}} \ar[r] \ar[d]
    & {A_{0}} \ar[d]\\
    {D^{1}} \ar[r]
    & {B \Period}
  }
\end{displaymath}
That is, for every such square we attach a $1$-cell to $A_{0}$, and we
use the bottom horizontal map of that square to map that attached
$1$-cell to $B$.

If $n > 1$ and we have constructed $A_{n-1}$ along with it's map to
$B$, we construct $A_{n}$ by constructing the pushout
\begin{displaymath}
  \xymatrix@=1em{
    {\hspace{-3.5em}
      \coprod_{\mappullback{(S^{n-1},A_{n-1})}{(S^{n-1},B)}{(D^{n},B)}}
      \hspace{-6em}S^{n-1}}
    \ar[rr] \ar[dd]
    && {A_{n-1}} \ar@{..>}[dl] \ar[dd]\\
    & {A_{n}} \ar@{..>}[dr]\\
    {\hspace{-4.5em}
      \coprod_{\mappullback{(S^{n-1},A_{n-1})}{(S^{n-1},B)}{(D^{n},B)}}
      \hspace{-6em}D^{n}}
    \ar@{..>}[ur] \ar[rr]
    && {B}
  }
\end{displaymath}
where $\mappullback{(S^{n-1},A_{n-1})}{(S^{n-1},B)}{(D^{n},B)}$ is the set
of commutative squares
\begin{displaymath}
  \xymatrix{
    {S^{n-1}} \ar[r] \ar[d]
    & {A_{n-1}} \ar[d]\\
    {D^{n}} \ar[r]
    & {B \Period}
  }
\end{displaymath}
That is, for every such square we attach an $n$-cell to $A_{n-1}$, and
we use the bottom horizontal map of that square to map that attached
$n$-cell to $B$.

To complete the construction we let $\Btilde = \colim_{n} A_{n}$, and
the map $A \to \Btilde$ is clearly a relative CW-complex.  We show
that the map $\Btilde \to B$ is a weak equivalence in
\secref{sec:weq}, and we show that the construction is natural in
\secref{sec:func}.

\subsection{The homotopy groups of the spaces in the construction}
\label{sec:weq}

\begin{notation}
  \label{not:connected}
  If $f\colon X \to Y$ is a map and $n \ge 0$, then we will say that
  $f$ is \emph{$n$-connected} if
  \begin{itemize}
  \item the set of path components of $X$ maps onto the set of path
    components of $Y$, and
  \item for every choice of basepoint in $X$ the induced map of
    homotopy groups (for $i > 0$) or sets (for $i = 0$) $\pi_{i}(X)
    \to \pi_{i}(Y)$ is an isomorphism for $i < n$ and an epimorphism
    for $i = n$.
  \end{itemize}
\end{notation}

\begin{lem}
  \label{lem:connected}
  For each $n \ge 0$ the map $A_{n} \to B$ is $n$-connected.
\end{lem}

\begin{proof}
  We will show inductively on $n$ that the map $A_{n} \to B$ is
  $n$-connected.

  The space $A_{0}$ was constructed to map onto $B$, and so the map
  $A_{0} \to B$ is $0$-connected

  The space $A_{1}$ was constructed by attaching $1$-cells to $A_{0}$
  that connected any pair of points in $A_{0}$ whose images were in
  the same path component of $B$; thus, the set of path components of
  $A_{1}$ maps isomorphically to the set of path components of $B$.
  In addition, a loop was wedged at every point of $A_{0}$ for every
  loop in $B$ at the image of that point; thus, for every basepoint of
  $A_{1}$, the fundamental group of $A_{1}$ maps epimorphically onto
  the fundamental group of $B$.  Thus, the map $A_{1} \to B$ is
  $1$-connected.

  Suppose now that $n > 1$ and that the map $A_{n-1} \to B$ is
  $(n-1)$-connected.  Since $A_{n}$ is constructed from $A_{n-1}$ by
  attaching $n$-cells, for every choice of basepoint we have
  $\pi_{i}(A_{n-1}) \iso \pi_{i}(A_{n})$ for $i < n-1$ and
  $\pi_{n-1}(A_{n})$ is a quotient of $\pi_{n-1}(A_{n-1})$.  For every
  map $\alpha\colon S^{{n-1}} \to A_{n-1}$ such that the composition
  with $A_{n-1} \to B$ is nullhomotopic, we've attached an $n$-cell,
  and so the composition $S^{n-1} \xrightarrow{\alpha} A_{n-1} \to
  A_{n}$ is nullhomotopic.  Thus, $\pi_{n-1}(A_{n}) \to \pi_{n-1}(B)$
  is an isomorphism for every choice of basepoint.  In addition, for
  every map $\beta\colon D^{n}/S^{{n-1}} \to B$ for which the image of
  the collapsed $S^{n-1}$ is in the image of $A_{n-1} \to B$, we've
  wedged on a copy of $D^{n}/S^{n-1}$ to $A_{n-1}$ and mapped it to
  $B$ using $\beta$, and so $\pi_{n}(A_{n}) \to \pi_{n}(B)$ is
  surjective for every choice of basepoint.  Thus, the map $A_{n} \to
  B$ is $n$-connected.  This completes the induction.
\end{proof}

We now let $\Btilde = \colim_{n}A_{n}$.  Since spheres and disks are
compact, every map from a sphere or disk to $\colim_{n}A_{n}$ factors
through some $A_{n}$, and so we have $\colim_{n}\pi_{i}A_{n} \iso
\pi_{i}\Btilde$ for $i \ge 0$.  Since the map $\pi_{i}A_{n} \to
\pi_{i}B$ is an isomorphism for $n > i$, the map $\pi_{i}\Btilde \to
\pi_{i}B$ is an isomorphism for $i \ge 0$, and so the map $\Btilde \to
B$ is a weak equivalence.

\subsection{The functoriality of the construction}
\label{sec:func}

We will now show that the construction of \secref{sec:construct} is
functorial, i.e., that if we have a commutative square
\begin{displaymath}
  \xymatrix{
    {A} \ar[r]^{f} \ar[d]
    & {A'} \ar[d]\\
    {B} \ar[r]_{g}
    & {B'}
  }
\end{displaymath}
and we apply the construction of \secref{sec:construct} to $A \to B$
to obtain $A \to \Btilde \to B$ and to $A' \to B'$ to obtain $A' \to
\Btilde' \to B'$, then there is a natural commutative diagram
\begin{displaymath}
  \xymatrix{
    {A} \ar[r]^{f} \ar[d]
    & {A'} \ar[d]\\
    {\Btilde} \ar[r]^{\gtilde} \ar[d]
    & {\Btilde'} \ar[d]\\
    {B} \ar[r]_{g}
    & {B' \Period}
  }
\end{displaymath}

We define $\gtilde$ by defining $f_{n}\colon A_{n} \to A_{n}'$
inductively on the constructions of $\Btilde$ and $\Btilde'$.

To begin, we have
\begin{displaymath}
  A_{0} = A_{-1} \amalg \Bigl(\coprod_{D^{0} \to B} D^{0}\Bigr)
  \qquad\text{and}\qquad
  A_{0}' = A_{-1}' \amalg \Bigl(\coprod_{D^{0} \to B'} D^{0}\Bigr)
\end{displaymath}
and we define $f_{0}\colon A_{0} \to A_{0}'$ by sending the copy of
$D^{0}$ indexed by $\alpha\colon D^{0} \to B$ to the copy of $D^{0}$
indexed by $g\circ\alpha\colon D^{0} \to B'$.

For the inductive step, suppose that $n > 0$ and that we've defined
$f_{n-1}\colon A_{n-1} \to A_{n-1}'$.  The space $A_{n}$ is
constructed by attaching an $n$-cell to $A_{n-1}$ for each commutative
square
\begin{displaymath}
  \xymatrix{
    {S^{n-1}} \ar[r]^{\alpha} \ar[d]
    & {A_{n-1}} \ar[d]\\
    {D^{n}} \ar[r]_{\beta}
    & {B}
  }
\end{displaymath}
We take the cell attached to $A_{n-1}$ by the map $\alpha$ to the cell
attached to $A_{n-1}'$ by the map $f_{n-1}\circ\alpha$ indexed by the
outer commutative rectangle
\begin{displaymath}
  \xymatrix{
    {S^{n-1}} \ar[r]^{\alpha} \ar[d]
    & {A_{n-1}} \ar[d] \ar[r]^{f_{n-1}}
    & {A_{n-1}'} \ar[d]\\
    {D^{n}} \ar[r]_{\beta}
    & {B} \ar[r]_{g}
    & {B'}
  }
\end{displaymath}
Doing that for each $n$-cell attached to $A_{n-1}$ defines
$f_{n}\colon A_{n} \to A_{n}'$.

That completes the induction, and we let $\gtilde\colon \Btilde \to
\Btilde'$ be $\colim_{n}f_{n}$.

\section{The proof of \thmref{thm:subcomp}}
\label{sec:subcomp}

We construct the factorization in \secref{sec:subcompconst}, show that
the map $\Btilde \to B$ is a weak equivalence in
\secref{sec:subcompwe}, and show that the construction is functorial
in \secref{sec:subcompfunc}.

\subsection{The construction}
\label{sec:subcompconst}

We use a modification of the construction of \secref{sec:construct}.
We construct $A_{0}$ exactly as in \secref{sec:construct}, but when $n
> 0$ and we are constructing $A_{n}$ out of $A_{n-1}$, we attach only
the $n$-cells indexed by commutative squares
\begin{displaymath}
  \xymatrix{
    {S^{n-1}} \ar[r]^{\alpha} \ar[d]
    & {A_{n-1}} \ar[d]\\
    {D^{n}} \ar[r]_{\beta}
    & {B}
  }
\end{displaymath}
for which $\alpha\colon S^{n-1} \to A_{n-1}$ is a cellular map.

\subsection{The homotopy groups of the spaces in the construction}
\label{sec:subcompwe}

\begin{lem}
  \label{lem:subcompconn}
  For each $n \ge 0$ the map $A_{n} \to B$ is $n$-connected.
\end{lem}

\begin{proof}
  We will show inductively on $n$ that the map $A_{n} \to B$ is
  $n$-connected.

  The space $A_{0}$ was constructed to map onto $B$, and so the map
  $A_{0} \to B$ is $0$-connected.

  The space $A_{1}$ was constructed by attaching $1$-cells to $A_{0}$
  that connected any pair of vertices in $A_{0}$ whose images were in
  the same path component of $B$; since every path component of
  $A_{0}$ contains at least one vertex, the set of path components of
  $A_{1}$ maps isomorphically to the set of path components of $B$.
  In addition, a loop was wedged at every vertex of $A_{0}$ for every
  loop in $B$ at the image of that vertex; since every path component
  of $B$ contains the image of a vertex of $A_{0}$, for every
  basepoint of $A_{1}$ the fundamental group of $A_{1}$ maps
  epimorphically onto the fundamental group of $B$.  Thus, the map
  $A_{1} \to B$ is $1$-connected.

  Suppose now that $n > 1$ and that the map $A_{n-1} \to B$ is
  $(n-1)$-connected.  Since $A_{n}$ is constructed from $A_{n-1}$ by
  attaching $n$-cells, for every choice of basepoint we have
  $\pi_{i}(A_{n-1}) \iso \pi_{i}(A_{n})$ for $i < n-1$ and
  $\pi_{n-1}(A_{n})$ is a quotient of $\pi_{n-1}(A_{n-1})$.  For every
  cellular map $\alpha\colon S^{n-1} \to A_{n-1}$ such that the
  composition with $A_{n-1} \to B$ is nullhomotopic, we've attached an
  $n$-cell, and so the composition $S^{n-1} \xrightarrow{\alpha}
  A_{n-1} \to A_{n}$ is nullhomotopic.  Since every map $S^{n-1} \to
  A_{n-1}$ is homotopic to a cellular map, $\pi_{n-1}(A_{n}) \to
  \pi_{n-1}(B)$ is an isomorphism for every choice of basepoint.  In
  addition, for every map $\beta\colon D^{n}/S^{{n-1}} \to B$ for
  which the image of the collapsed $S^{n-1}$ is in the image of a
  vertex of $A_{n-1}$, we've wedged on a copy of $D^{n}/S^{n-1}$ to
  that vertex of $A_{n-1}$ and mapped it to $B$ using $\beta$; since
  every path component of $B$ is in the image of a vertex of
  $A_{n-1}$, $\pi_{n}(A_{n}) \to \pi_{n}(B)$ is surjective for every
  choice of basepoint.  Thus, the map $A_{n} \to B$ is $n$-connected.
  This completes the induction.
\end{proof}

We now let $\Btilde = \colim_{n}A_{n}$.  Since spheres and disks are
compact, every map from a sphere or disk to $\colim_{n}A_{n}$ factors
through some $A_{n}$, and so we have $\colim_{n}\pi_{i}A_{n} \iso
\pi_{i}\Btilde$ for $i \ge 0$.  Since the map $\pi_{i}A_{n} \to
\pi_{i}B$ is an isomorphism for $n > i$, the map $\pi_{i}\Btilde \to
\pi_{i}B$ is an isomorphism for $i \ge 0$, and so the map $\Btilde \to
B$ is a weak equivalence.

\subsection{The functoriality of the construction}
\label{sec:subcompfunc}

We will now show that the construction of \secref{sec:subcompconst} is
functorial, i.e., that if we have a commutative square
\begin{displaymath}
  \xymatrix{
    {A} \ar[r]^{f} \ar[d]
    & {A'} \ar[d]\\
    {B} \ar[r]_{g}
    & {B'}
  }
\end{displaymath}
in which $f\colon A \to A'$ is a cellular map and we apply the
construction of \secref{sec:subcompconst} to $A \to B$ to obtain $A
\to \Btilde \to B$ and to $A' \to B'$ to obtain $A' \to \Btilde' \to
B'$, then there is a natural commutative diagram
\begin{displaymath}
  \xymatrix{
    {A} \ar[r]^{f} \ar[d]
    & {A'} \ar[d]\\
    {\Btilde} \ar[r]^{\gtilde} \ar[d]
    & {\Btilde'} \ar[d]\\
    {B} \ar[r]_{g}
    & {B' \Period}
  }
\end{displaymath}

We define $\gtilde$ by defining $f_{n}\colon A_{n} \to A_{n}'$
inductively on the constructions of $\Btilde$ and $\Btilde'$.  Since
each $f_{n}\colon A_{n} \to A_{n}'$ is a cellular map, the composition
of a cellular map $\alpha\colon S^{n-1} \to A_{n}'$ with
$f_{n-1}\colon A_{n-1} \to A_{n-1}'$ is also cellular, and so we have
an induced map $f_{n}\colon A_{n} \to A_{n}'$.  Thus, the induction
goes through, and we let $\gtilde\colon \Btilde \to \Btilde'$ be
$\colim_{n}f_{n}$.

\section{Proof of \propref{prop:SameCW}}
\label{sec:SameCW}

Since we are factorizing the map $\emptyset \to B$, in the sequence
$A_{-1} \to A_{0} \to A_{1} \to \cdots$ whose colimit is $\Btilde$
(see \secref{sec:construct}) the space $A_{-1}$ is empty.  Thus, for
each $n \ge 0$ the space $A_{n}$ is an $n$-dimensional CW-complex, and
so \emph{every} map $S^{n} \to A_{n}$ is a cellular map.  Thus, the
sequence constructed in \secref{sec:subcompconst} is exactly the same
as the sequence constructed in \secref{sec:construct}, and so their
colimits are the same.

\section{Proof of \thmref{thm:subspace}}
\label{sec:subspace}

We will show by induction that in the diagram
\begin{displaymath}
  \xymatrix{
    {\emptyset = A_{-1}}  \ar[r]
    & {A_{0}} \ar[r] \ar[d]
    & {A_{1}} \ar[r] \ar[d]
    & {A_{2}} \ar[r] \ar[d]
    & {\cdots}\\
    {\emptyset = A_{-1}'}  \ar[r]
    & {A_{0}'} \ar[r]
    & {A_{1}'} \ar[r]
    & {A_{2}'} \ar[r]
    & {\cdots}\\
  }
\end{displaymath}
used to construct $\Btilde \to \Btilde'$, each map $A_{n} \to A_{n}'$
is an inclusion of a subcomplex.  The induction is begun because
$A_{0}$ has one point for every point of $B$ and $A_{0}'$ has one
point for every point of $B'$.

Now assume that $n > 0$ and that $A_{n-1} \to A_{n-1}'$ is an
inclusion of a subcomplex. Since the map $B \to B'$ is also an
inclusion, the set of $n$-cells to be attached to $A_{n-1}$ is a
subset of the set of $n$-cells to be attached to $A_{n-1}'$, and so
$A_{n} \to A_{n}'$ will also be an inclusion of a subcomplex.

\section{The proof of \thmref{thm:Intersect}}
\label{sec:Intersect}

Let $X_{S} = \cap_{s\in S} X_{s}$.
\begin{itemize}
\item Let $\emptyset = A_{-1} \to A_{0} \to A_{1} \to \cdots$ be the
  sequence created in the proof of \thmref{thm:cwapprox} whose colimit
  is $\CW(X)$,
\item let $\emptyset = A^{S}_{-1} \to A^{S}_{0} \to A^{S}_{1} \to
  \cdots$ be the sequence created in the proof of
  \thmref{thm:cwapprox} whose colimit is $\CW(X_{S})$, and
\item for each $s \in S$ let $\emptyset = A^{s}_{-1} \to A^{s}_{0} \to
  A^{s}_{1} \to \cdots$ be the sequence created in the proof of
  \thmref{thm:cwapprox} whose colimit is $\CW(X_{s})$.
\end{itemize}
The proof of \thmref{thm:subspace} shows that $A^{S}_{n}$ and
$A^{s}_{n}$ are subcomplexes of $A_{n}$ for all $s \in S$ and $n \ge
0$; we will show by induction that $A^{S}_{n} = \cap_{s\in S}
A^{s}_{n}$ for all $n \ge 0$.

Since $A^{S}_{0}$ is discrete with one point for each point of $X_{S}$
and for all $s \in S$ the space $A^{s}_{0}$ is discrete with one point
for each point of $X_{s}$, we have $A^{S}_{0} = \cap_{s\in S}
A^{s}_{0}$.

Assume now that $n > 0$ and $A^{S}_{n-1} = \cap_{s\in S}
A^{s}_{n-1}$.  The space $A^{S}_{n}$ is constructed by attaching an
$n$-cell to $A^{S}_{n-1}$ for each commutative square
\begin{displaymath}
  \xymatrix{
    {S^{n-1}} \ar[r] \ar[d]
    & *+[r]{A^{S}_{n-1} = \cap_{s\in S} A^{s}_{n-1}} \ar[d]\\
    {D^{n}} \ar[r]
    & *+[r]{X_{S} = \cap_{s\in S} X_{s}}
  }
\end{displaymath}
Since the maps $A^{S}_{n-1} \to A^{s}_{n-1}$ and $X_{S} \to X_{s}$ are
inclusions for all $s \in S$, each such $n$-cell corresponds to a
unique $n$-cell in $\cap_{s\in S} A^{s}_{n}$, i.e., the map $A^{S}_{n}
\to \cap_{s\in S} A^{s}_{n}$ is an injection.

To see that the map $A^{S}_{n} \to \cap_{s\in S} A^{s}_{n}$ is a
surjection, let
\begin{displaymath}
  \left.
  \vcenter{
    \xymatrix{
      {S^{n-1}} \ar[r]^-{\alpha_{s}} \ar[d]
      & {A^{s}_{n-1}} \ar[d]\\
      {D^{n}} \ar[r]_{\beta_{s}}
      & {X_{s}}
    }
  }
  \right\}
  \text{for every $s \in S$}
\end{displaymath}
index $n$-cells of the $A^{s}_{n}$ that together define an $n$-cell of
$\cap_{s\in S} A^{s}_{n}$.  Since the maps $A^{s}_{n-1} \to A_{n-1}$
and $X_{s} \to X$ are all inclusions, the compositions $S^{n-1}
\xrightarrow{\alpha_{s}} A^{s}_{n-1} \to A_{n-1}$ are all equal and
the compositions $D^{n} \xrightarrow{\beta_{s}} X_{s} \to X$ are all
equal, and the diagram
\begin{displaymath}
  \xymatrix{
    {S^{n-1}} \ar[r]^-{\alpha_{s}} \ar[d]
    & {A^{s}_{n-1}} \ar[r]
    & {A_{n-1}} \ar[d]\\
    {D^{n}} \ar[r]_-{\beta_{s}}
    & {X_{s}} \ar[r]
    & {X}
  }
\end{displaymath}
(for any $s \in S$; the upper and lower compositions are all the same)
indexes an $n$-cell that was attached to $A_{n-1}$ when creating
$A_{n}$.  Since the upper composition factors uniquely through
$\cap_{s\in S} A^{s}_{n-1}$ and the lower composition factors uniquely
through $X_{S} = \cap_{s\in S} X_{s}$, those factorizations index an
$n$-cell that was attached to $A^{S}_{n-1}$ when creating $A^{S}_{n}$,
and that $n$-cell maps to our $n$-cell of $\cap_{s\in S} A^{s}_{n}$.



\end{document}